\documentclass[12pt]{amsart}
\usepackage{amsmath,amssymb}
\usepackage{graphicx}
\usepackage{latexsym}
\usepackage{pinlabel}
\usepackage{verbatim}
\usepackage[all]{xy}

\newtheorem{thm}{Theorem}[section]

\newtheorem*{fscriterion}{Fiber Sum Criterion}

\newtheorem{theorem}[thm]{Theorem}
\newtheorem{lemma}[thm]{Lemma}

\newtheorem{proposition}[thm]{Proposition}

\theoremstyle{definition}

\newcommand{\Z}{\mathbb{Z}}

\newcommand{\nc}{\newcommand}
\nc{\dmo}{\DeclareMathOperator}

\dmo{\MCG}{Mod}
\dmo{\Diff}{Diff}
\dmo{\Dehn}{Dehn}
\dmo{\I}{\mathcal{I}}
\dmo{\Sp}{Sp}
\dmo{\PB}{PB}
\dmo{\B}{B}
\dmo{\D}{D}

\nc{\Push}{\mathcal{P}ush}
\nc{\N}{\mathcal{N}}
\nc{\K}{\mathcal{K}}
\dmo{\Out}{Out}

\dmo{\UT}{UT}

\nc{\margin}[1]{\marginpar{\scriptsize #1}}

\nc{\p}[1]{\medskip\paragraph{{\bf #1.}}}

\begin{document}

\title[Lefschetz fibrations and Torelli groups]
{Lefschetz fibrations and Torelli groups}

\author[R. \.{I}. Baykur]{R. \.{I}nan\c{c} Baykur}
\address{Max Planck Institute for Mathematics, Bonn, Germany \newline
\indent Department of Mathematics, Brandeis University, Waltham, MA, USA}
\email{baykur@mpim-bonn.mpg.de, baykur@brandeis.edu}

\author[Dan Margalit]{Dan Margalit}
\address{School of Mathematics, Georgia Institute of Technology, Atlanta, GA, USA}
\email{margalit@math.gatech.edu}

\begin{abstract}
For each $g \geq 3$ and $h \geq 2$, we explicitly construct \linebreak (1) fiber sum indecomposable relatively minimal genus $g$ Lefschetz fibrations over genus $h$ surfaces whose monodromies lie in the Torelli group, (2) fiber sum indecomposable genus $g$ surface bundles over genus $h$ surfaces whose monodromies are in the Torelli group (provided $g \geq 4$), and (3) infinitely many genus $g$ Lefschetz fibrations over genus $h$ surfaces that are not fiber sums of holomorphic ones. 
\end{abstract}

\thanks{The first author was partially supported by the NSF grant DMS-0906912. The second author was partially supported by an NSF CAREER grant and a fellowship from the Sloan Foundation.}

\maketitle

\setcounter{secnumdepth}{2}
\setcounter{section}{0}

\section{Introduction}

Given a Lefschetz fibration $f\colon X \to \Sigma$ on a closed oriented \linebreak $4$-manifold $X$ with a regular fiber $F$ and critical locus $\text{Crit}(f)$, there is an associated monodromy homomorphism $\mu$ from $\pi_1(\Sigma \setminus \text{Crit}(f))$ to the mapping class group of $F$, which determines the topology of the pair $(X,f)$. A more restricted but still useful piece of information is the image of $\mu$, the monodromy group of $f$. For example, there are various restrictions on the topology of a Lefschetz fibration whose monodromy group lies in the hyperelliptic mapping class group; see \cite{EndoHyperelliptic, SiebertTianHyperelliptic}.  The purpose of this article is to study Lefschetz fibrations with monodromy group in the Torelli group, the kernel of the action of the mapping class group of $F$ on $H_1(F;\Z)$ (Theorems~\ref{mainthm1} and~\ref{mainthm3}). As a byproduct we will construct families of Lefschetz fibrations that are not fiber sums of holomorphic ones (Theorem~\ref{mainthm2}).

In addition to having monodromy in the Torelli group, the Lefschetz fibrations we construct will have the additional property of fiber sum indecomposability: whenever we write them as a fiber sum of two Lefschetz fibrations, one is a trivial bundle over $S^2$.  Such fibrations can be regarded as the irreducible building blocks of Lefschetz fibrations.  

In what follows, we reserve the term Lefschetz fibration for the case where the critical locus is nonempty, so as to distinguish this from the surface bundle case.

\begin{thm} \label{mainthm1}
For each pair of integers $g \geq 3$ and $h \geq 2$, there are fiber sum indecomposable relatively minimal genus $g$ Lefschetz fibrations over genus $h$ surfaces whose monodromies lie in the Torelli group.
\end{thm}

Theorem~\ref{mainthm1} is proven by explicitly constructing the desired Lefschetz fibrations. In the case $h=2$ these fibrations are prescribed by the monodromy factorizations 
\[ T_{c}^2 \, [T_{\alpha_1^-}T_{\alpha_1^+}^{-1} , T_{\beta_1^-}T_{\beta_1^+}^{-1} ] [T_{\alpha_2^-}T_{\alpha_2^+}^{-1} , T_{\beta_2^-}T_{\beta_2^+}^{-1} ] = 1
\] 
where where $T_x$ denotes the right Dehn twist about $x$ and the curves $c, \alpha_j^+, \alpha_j^-, \beta_j^+, \beta_j^-$ are as in Figure~\ref{figure:pushrelation} (to realize the curves in a surface of genus $g \geq 3$, glue a genus $g-2$ surface along $c$).  Our conventions here and throughout the paper are that $[a,b]$ denotes $aba^{-1}b^{-1}$ and that mapping classes are applied right to left.  Our fibrations for $h \geq 3$ are then obtained by taking appropriate pullbacks of these fibrations.

\begin{figure}
\labellist
\small\hair 2pt
\pinlabel {$\alpha_1^-$} [] at 95 150 
\pinlabel {$\alpha_1^+$} [] at 95 110 
\pinlabel {$\beta_1^-$} [] at 265 150 
\pinlabel {$\beta_1^+$} [] at 305 150 
\pinlabel {$\alpha_2^+$} [] at 75 60 
\pinlabel {$\alpha_2^-$} [] at 75 25 
\pinlabel {$\beta_2^+$} [] at 227 20 
\pinlabel {$\beta_2^-$} [] at 274 40
\pinlabel {$c$} [] at 173 160
\endlabellist
\centering \includegraphics[scale=.9]{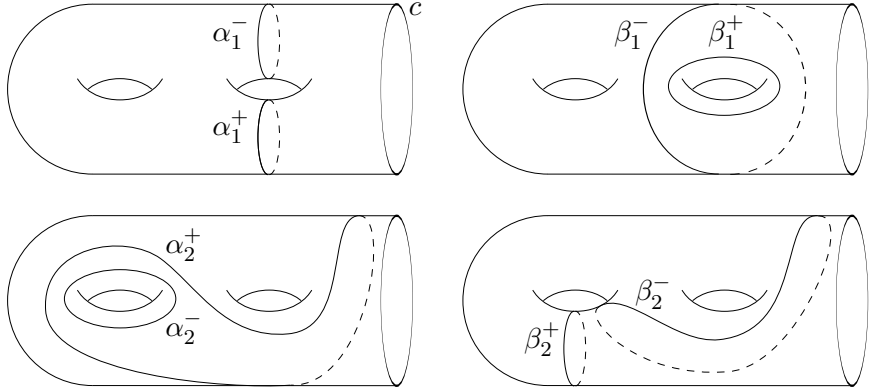}
\caption{The curves used in Theorem~\ref{mainthm1}.}
\label{figure:pushrelation}
\end{figure}

The fiber sum indecomposability of our fibrations alone contrasts with the situation for the lower genera fibrations: as shown by Matsumoto in \cite{Matsumoto1}, any genus one Lefschetz fibration is fiberwise diffeomorphic to the fiber sum of standard ones, namely the elliptic surfaces $E(n)$ and the trivial surface bundle over a genus $h$ surface. 

The hypotheses on the fiber and base genera in Theorem~\ref{mainthm1} are necessary: as we show in Proposition~\ref{SmallGenera}, there are no relatively minimal Lefschetz fibrations with monodromy group contained in the Torelli group if $g \leq 2$ or $h \leq 1$; the $h=0$ case is due to Smith \cite{SmithTorelli, ABKPTorelli}. Relative minimality is also a necessary assumption in Theorem~\ref{mainthm1}, for otherwise one can simply blow up on a fiber of any trivial surface bundle over a surface to produce such examples.

In the course of the proof of Theorem~\ref{mainthm1} we prove that the commutator length in the Torelli group (or the mapping class group) of a Dehn twist about a curve cutting off a genus 2 subsurface with one boundary component is 2.  This result is proved in the same spirit as in \cite{KorkmazOzbagciCL, KorkmazCL, EndoKotschickCL, BaykurKorkmazMonden}, where similar calculations were carried out for other powers of Dehn twists in the whole mapping class group. 

For completeness, we also give a version of Theorem~\ref{mainthm1} for surface bundles over surfaces.

\begin{thm} \label{mainthm3}
For each pair $g \geq 4$ and $h \geq 2$, there are fiber sum indecomposable genus $g$ surface bundles over genus $h$ surfaces whose monodromy groups lie in the Torelli group. 
\end{thm}

The proof of Theorem~\ref{mainthm3} is by explicit construction.  We give the first explicit examples of surface subgroups of the Torelli group, although nonexplicit examples are readily available \cite{CLM,Koberda,CrispFarb}. It is easy to generate explicit examples for $h=1$: any two Dehn twists along distinct, disjoint essential separating curves generate a subgroup of the Torelli group isomorphic to $\pi_1(T^2) \cong \Z^2$.

It is certainly necessary in Theorem~\ref{mainthm3} that $g \geq 3$, since the Torelli group is trivial for $g=1,2$ and is a free group for $g=2$. It is not immediately clear from our work if the $g=3$ case is achievable.  

When stripped of the fiber sum indecomposability property, the existence of the Lefschetz fibrations and surface bundles in Theorems~\ref{mainthm1} and~\ref{mainthm3} is no surprise. The trivial surface bundles already have monodromy in the Torelli group, and nontrivial examples are very easy to generate. As for Lefschetz fibrations, it is known that the commutator subgroup of the Torelli group has finite index in the subgroup generated by Dehn twists about separating curves \cite{Johnson}, and so some high enough power of a product of Dehn twists about separating curves is equal to a product of commutators in the Torelli group. Nevertheless, this fact does not provide any explicit information on the number of commutators.  Therefore, it does not allow us to explicitly realize all pairs of $g$ and $h$ as in our theorems.

When $g \geq 9$, Theorem~\ref{mainthm3} is a special case of another theorem we prove, namely, that there exist fiber sum indecomposable surface bundles over surfaces whose monodromy groups lie in any given term of the Johnson filtration; see Section~\ref{Sec: SBconstruction}. Again, we give explicit examples; the existence of such bundles was known to Crisp and Farb \cite{CrispFarb}.

\smallskip
We now turn to the following question: \emph{Is every Lefschetz fibration a fiber sum of holomorphic Lefschetz fibrations?} This question arose because the first examples of Lefschetz fibrations over the $2$-sphere that could not be made holomorphic for any choice of complex structures on the total space and the base were obtained by fiber summing holomorphic Lefschetz fibrations with twisted gluings \cite{OzbagciStipsiczHolomorphic, FintushelSternHolomorphic, KorkmazHolomorphic}. For Lefschetz fibrations over $S^2$, the question was settled negatively by Stipsicz \cite{StipsiczIndecomposable} and Smith \cite{SmithIndecomposable} independently. By modifying the Lefschetz fibrations we obtained in Theorem~\ref{mainthm1} we extend their result to Lefschetz fibrations over surfaces of higher genera.

\begin{thm} \label{mainthm2}
For each pair $g \geq 3$ and $h \geq 2$, there are relatively minimal genus $g$ Lefschetz fibrations over genus $h$ surfaces on infinitely many pairwise homotopy inequivalent $4$-manifolds that are not fiber sums of holomorphic Lefschetz fibrations.  
\end{thm}

The proofs of Stipsicz and Smith for the $h=0$ case fails when the base genus is positive; see the discussion at the end of Section~\ref{Sec: LFconstruction}. Our proof is obtained via a direct analysis of the monodromies, making use of an algebraic characterization of fiber sum decomposability we provide at the end of Section~\ref{Sec: Background}.  

\p{Acknowledgments} We would like to thank Benson Farb and Sang-hyun Kim for helpful conversations.

\section{Background} \label{Sec: Background}

In this section we present some background material on mapping class groups, Lefschetz fibrations, and surface bundles.

\subsection{Mapping class groups and Torelli groups} Let $F$ denote a compact, connected, oriented surface with a finite set of marked points in its interior. The \emph{mapping class group} of $F$, denoted $\MCG(F)$, is the group of isotopy classes of orientation-preserving self-diffeomorphisms of $F$ that fix $\partial F$ pointwise and preserve the set of marked points.
 
Denote by $\Sigma_g$ a closed, connected, oriented surface of genus $g$ and by $\Sigma_g^1$ the surface obtained from $\Sigma_g$ by deleting the interior of an embedded disk.  Let $F$ be either $\Sigma_g$ or $\Sigma_g^1$.  The \emph{Torelli group} of $F$, denoted $\I(F)$, is the kernel of the action of $\MCG(F)$ on $H_1(F;\Z)$.

\subsection{Lefschetz fibrations and surface bundles.} A \emph{Lefschetz fibration} $(X,f)$ of a smooth $4$-manifold $X$ is a surjection $f$ to a closed oriented surface $\Sigma$ that is a submersion on the complement of finitely many points $p_i$, at which there are local complex coordinates (compatible with the orientations on $X$ and $\Sigma$) with respect to which the map takes the form $(z_1, z_2) \mapsto z_1 z_2$. When there are no critical points, $(X,f)$ is a surface bundle over a surface.  We say that $(X,f)$ is a genus $g$ Lefschetz fibration (or a surface bundle) over a genus $h$ surface if the genus of a regular fiber $F$ is $g$ and the genus of the base $\Sigma$ is $h$.

Again, in this paper, we only use the term Lefschetz fibration when the set of critical points $\{p_i\}$ is nonempty.  We will moreover assume that all the points $p_i$ lie in distinct fibers; this can always be achieved after a small perturbation. Lastly, the fibration is called \emph{relatively minimal} if there are no spheres of self-intersection $-1$ contained in the fibers. 

 \subsection{Monodromy factorizations.}\label{sec:monodromy} The topology of a Lefschetz fibration $f\colon X \to \Sigma$ with regular fiber $F$ and critical points $p_i$ is determined by the monodromy homomorphism\footnote{A technical point we will suppress is that the monodromy is really an anti-homomorphism, since elements of $\pi_1(\Sigma \setminus \{f(p_i)\})$ are written left-to-right and elements of $\MCG(F)$ are written right-to-left.}
\[ \mu : \pi_1(\Sigma \setminus \{f(p_i)\}) \to \MCG(F) \]
 \cite{Kas, Matsumoto}. The monodromy $\mu$ maps an element of $\pi_1(\Sigma \setminus \{f(p_i)\})$ encircling a single critical value $f(p_i)$ in a  counterclockwise fashion to a positive Dehn twist in $\MCG(F)$. If $D$ is a $2$-disk that contains all of the critical values, then $f$ restricts to a surface bundle over $\Sigma \setminus D$. 

We can choose standard generators $\alpha_j, \beta_j$ for $\pi_1(\Sigma \setminus D)$ and standard generators of $\gamma_\ell$ of $\pi_1(D\setminus\{ f(p_1),\dots,f(p_k)\})$ so that
\[ [\alpha_1, \beta_1] \cdots [\alpha_h, \beta_h] = \gamma_1 \cdots \gamma_k . \]
A Lefschetz fibration $X \to \Sigma_h$ is completely determined by the images of the $\alpha_j$, $\beta_j$, and $\gamma_\ell$ under the monodromy $\mu$.  Since the above relation is the only defining relation for $\pi_1(\Sigma_h\setminus\{ f(p_1),\dots,f(p_k)\})$, genus $g$ Lefschetz fibrations $f : X \to \Sigma_h$ are completely determined by choices of $\mu(\alpha_j),\mu(\beta_j) \in \MCG(\Sigma_g)$ and choices of positive Dehn twists $T_{c_\ell}=\mu(\gamma_\ell) \in \MCG(\Sigma_g)$ satisfying the relation
\[ T_{c_k} \cdots  T_{c_1}  \, [\mu(\alpha_1)^{-1}, \mu(\beta_1)^{-1}] \cdots [\mu(\alpha_h)^{-1}, \mu(\beta_h)^{-1}]= 1 .\]
Such an expression is called a \emph{monodromy factorization} for $(X,f)$. Conversely, whenever we have such a factorization, we can build a genus $g$ Lefschetz fibration $(X,f)$ over a genus $h$ surface with $k$ critical points. 

A Dehn twist $T_c$ is contained in the Torelli group if and only if $c$ is a separating curve. So the monodromy $\mu$ of a Lefschetz fibration $(X,f)$ lies in the Torelli group $\I(F) < \MCG(F)$ if and only if there is a monodromy factorization as above where $\mu(\alpha_j), \mu(\beta_j)$ are in $\I(F)$ and each $\mu(\gamma_\ell)$ is positive Dehn twist along a separating curve. Moreover, it is not hard to see that a monodromy factorization contains a Dehn twist along a nullhomotopic curve if and only if there is a self-intersection $-1$ sphere in a fiber. 

Analogous statements hold mutatis mutandis for surface bundles over surfaces where we simply drop from the discussion the fundamental group generators enclosing the critical points.

\subsection{Pullbacks and monodromies} \label{pullbacksubsection} Given a genus $g$ Lefschetz fibration $f\colon X \to \Sigma$ and a covering $p: \widetilde{\Sigma} \to \Sigma$, we can define the pullback Lefschetz fibration $\widetilde{f}\colon \widetilde{X} \to \widetilde{\Sigma}$ in the usual way.  The pullback of a relatively minimal Lefschetz fibration is again a relatively minimal Lefschetz fibration of the same fiber genus.  

Say the critical values of $f$ are $q_1, \ldots, q_n$, and denote by $p'$ the restriction of $p$ to $\widetilde{\Sigma} \setminus p^{-1}(\{ q_i \})$.  One monodromy factorization of the pullback is 
\[ \widetilde{\mu} = \mu \circ p'_\star : \pi_1\left(\widetilde{\Sigma} \setminus p^{-1}(\{ q_i \})\right) \to \MCG(\Sigma_g) ,\]
where $\mu$ is a monodromy factorization for $f$. If the degree of $p$ is $m$ then $\widetilde f$ has $m n$ critical points.   The free homotopy class of a counterclockwise simple loop around one point of $p^{-1}(q_i)$ maps to the free homotopy class of a counterclockwise simple loop around $q_i$, and therefore the monodromy around each point of $p^{-1}(q_i)$ in the pullback is conjugate to the monodromy around $q_i$ in the original fibration.  Since the monodromy around $q_i$ is a nontrivial positive Dehn twist, we can see directly from the monodromy description of the pullback that the monodromy around each point of $p^{-1}(q_i)$ is a nontrivial positive Dehn twist.

\subsection{Fiber sum indecomposability.}\label{sec:fiber sum} A common way to construct new Lefschetz fibrations from old ones is the \emph{fiber sum} operation, defined as follows. Let $(X_1, f_1)$ and $(X_2,f_2)$ be genus $g$ Lefschetz fibrations over surfaces of genus $h_1$ and $h_2$, respectively, with regular fibers $F_1$ and $F_2$. The \emph{fiber sum} of $(X_1, f_1)$ and $(X_2,f_2)$ is a genus $g$ Lefschetz fibration over a surface of genus $h_1+h_2$ obtained by removing a fibered tubular neighborhood of each $F_i$ and then identifying the resulting boundaries via any fiber-preserving, orientation-reversing diffeomorphism.

As in the introduction, a Lefschetz fibration $(X,f)$ is fiber sum indecomposable if it cannot be expressed as a fiber sum of any two Lefschetz fibrations where neither is a trivial bundle over $S^2$. Again, such fibrations can be regarded as the irreducible building blocks of Lefschetz fibrations.

We will now give an algebraic interpretation of the fiber sum operation in terms of monodromies and present an algebraic criterion that obstructs the presence of nontrivial fiber sum decompositions. Our discussion here extends the one in our earlier work on surface bundles \cite{BaykurMargalit}. 

Let $(X_1,f_1)$ and $(X_2,f_2)$ be genus $g$ Lefchetz fibrations over surfaces of genus $h_1$ and $h_2$, respectively.  Let $\mu_i \colon \pi_1(\Sigma_i \setminus \text{Crit}(f_i)) \to \MCG(\Sigma_g)$ be the monodromy of the fibration $f_i$, for $i=1,2$.  There is an induced homomorphism
\[ \mu_1 \ast \mu_2 : \pi_1 (\Sigma_1 \setminus \text{Crit}(f_1)) \ast \pi_1 (\Sigma_2 \setminus \text{Crit}(f_2)) \to \MCG(\Sigma_g). \]
Let $\Sigma$ be a surface of genus $h=h_1+h_2$ obtained by taking the connected sum of the $\Sigma_i$, and let $\gamma$ denote the simple closed curve in $\Sigma$ along which the $\Sigma_i \setminus D^2$ are glued. Base $\pi_1(\Sigma)$ at a point of $\gamma$.  There is a homomorphism
\[ \pi_1(\Sigma \setminus \text{Crit}(f_i)) \to \pi_1(\Sigma_1 \setminus \text{Crit}(f_1)) \ast \pi_1(\Sigma_2 \setminus \text{Crit}(f_2)) \]
induced by collapsing $\gamma$ to a point. The monodromy of the fiber sum of $(X_1,f_1)$ and $(X_2,f_2)$ is induced by postcomposing the above map with $\mu_1 \ast \mu_2$.

We conclude that a Lefschetz fibration is fiber sum indecomposable if and only if its monodromy does not decompose into a nontrivial free product of two monodromies as above. More precisely:

\begin{fscriterion}
A Lefschetz fibration is fiber sum decomposable if and only if the kernel of the monodromy contains a nontrivial separating simple closed curve.
\end{fscriterion}

We can restate the Fiber Sum Criterion in terms of the monodromy factorization: a genus $g$ Lefschetz fibration $f$ over $\Sigma_h$ with $\ell$ critical points and monodromy $\mu$ is fiber sum decomposable if and only if there is a choice of generators $\alpha_j, \beta_j, \gamma_\ell$ for $\pi_1(\Sigma_h \setminus \text{Crit}(f))$ so that
\[ T_{c_n} \cdots T_{c_1} \, [\mu(\alpha_1)^{-1},\mu(\beta_1)^{-1}]\cdots[\mu(\alpha_m)^{-1},\mu(\beta_m)^{-1}]  =  1 \] 
in $\MCG(\Sigma_g)$, where $1 \leq 2m+n < 2h+ \ell$. In other words, a Lefschetz fibration is fiber sum decomposable if and only if it can be prescribed by a factorization of the identity in $\MCG(\Sigma_g)$ that contains a nontrivial and proper subfactorization of the identity. Recall that the fiber sum operation involves a choice of gluing map, which amounts to conjugating such a subfactorization by an element in $\MCG(\Sigma_g)$; this clearly preserves the property of having a proper subfactorization of the identity.

\section{Construction of Lefschetz fibrations} \label{Sec: LFconstruction}

In this section we will prove Theorems \ref{mainthm1} and~\ref{mainthm2}, that is, we will give explicit examples of fiber sum indecomposable Lefschetz fibrations with monodromy in the Torelli group, and then we will modify our construction to give our Lefschetz fibrations that are not fiber sums of holomorphic ones.  We begin by describing the key ingredient, unit tangent bundle subgroups of the mapping class group.

\subsection{Unit tangent bundle subgroups}

Let $2 \leq g_0 < g$.  Denote by $\UT(\Sigma_{g_0})$ the unit tangent bundle of $\Sigma_{g_0}$.   In this section we will construct an embedding of $\pi_1(\UT(\Sigma_{g_0}))$ in $\I(\Sigma_g)$.

First, the inclusion $\Sigma_{g_0}^1 \to \Sigma_g$ induces an injective homomorphism
\[ \MCG(\Sigma_{g_0}^1) \to \MCG(\Sigma_g); \]
see \cite[Theorem 3.18]{FarbMargalit}.  Since $H_1(\Sigma_{g_0}^1;\Z)$ injects into $H_1(\Sigma_g;\Z)$, the last homomorphism restricts to
\[ \I(\Sigma_{g_0}^1) \to \I(\Sigma_g). \]
Thus, it suffices to find a subgroup of $\I(\Sigma_{g_0}^1)$ isomorphic to $\pi_1(\UT(\Sigma_{g_0}))$.

There is a short exact sequence
\[ 1 \to \pi_1(\UT(\Sigma_{g_0})) \stackrel{\Push}{\to} \MCG(\Sigma_{g_0}^1) \to \MCG(\Sigma_{g_0}) \to 1;\]
see \cite[Section 4.2]{FarbMargalit}.

The group $\pi_1(\UT(\Sigma_{g_0}))$ has the presentation
\[ \langle \widetilde \alpha_1,\widetilde \beta_1,\dots,\widetilde \alpha_{g_0},\widetilde \beta_{g_0},t \mid [\widetilde \alpha_1,\widetilde \beta_1]\cdots [\widetilde \alpha_{g_0},\widetilde \beta_{g_0}]=t^{2g_0-2}, [\widetilde \alpha_i,t]=[\widetilde \beta_i,t]=1 \rangle. \] 
Here $t$ is the simple loop contained in the fiber and the other generators are arbitrarily chosen lifts of the standard generators for $\pi_1(\Sigma_{g_0})$.

We can explicitly describe the image of each generator of $\pi_1(\UT(\Sigma_{g_0}))$ in $\MCG(\Sigma_{g_0}^1)$.  First, the central element $t$ maps to the Dehn twist $T_c$, where $c$ is the boundary of $\Sigma_{g_0}^1$.  Next, let $\alpha$ denote the image of some $\widetilde \alpha \in \{\widetilde \alpha_i, \widetilde \beta_i\}$ in the fundamental group $\pi_1(\Sigma_{g_0})$ based at a point on the boundary.  We represent $\alpha$ by a simple loop in $\Sigma_{g_0}$ that is disjoint from the boundary away from its endpoints.  A closed regular neighborhood of $\alpha \cup c$ is homeomorphic to a pair of pants, that is, a sphere with three boundary components.  One of the boundary components is $c$.   We denote the two boundary components lying to the right and to the left of $\alpha$ by $\alpha^+$ and $\alpha^-$.  Finally, we have\footnote{Again, we have the situation where $\Push$ is an anti-homomorphism.}  
\[ \Push(\widetilde \alpha) = T_{\alpha^+}T_{\alpha^-}^{-1}. \]
In particular, we observe that, since $\alpha^+$ and $\alpha^-$ are homologous in $\Sigma_{g_0}^1$, the image of $\pi_1(\UT(\Sigma_{g_0}))$ in $\MCG(\Sigma_{g_0}^1)$ lies in $\I(\Sigma_{g_0}^1)$.  

Combining all of the above observations, we deduce the following fact.

\begin{figure}
\labellist
\small\hair 2pt
\pinlabel {$\alpha_1$} [] at 292 110 
\pinlabel {$\beta_1$} [] at 220 112 
\pinlabel {$\alpha_2$} [] at 95 122 
\pinlabel {$\beta_2$} [] at 140 40 
\endlabellist
\centering \includegraphics[scale=.9]{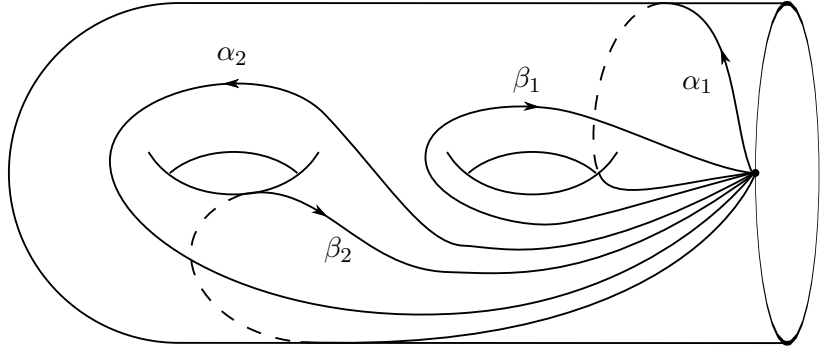}
\caption{Standard generators for $\pi_1(\Sigma_2)$}
\label{figure:generators}
\end{figure}

\begin{lemma}
\label{lemma:TorelliFactorization}
Let $g \geq 3$ and $2 \leq g_0 < g$.   Let $c$ be a separating curve in $\Sigma_g$ that cuts off a genus $g_0$ subsurface $\Sigma_{g_0}^1$, and choose standard generators $\{\widetilde \alpha_i,\widetilde \beta_i\}$ for $\pi_1(\UT(\Sigma_{g_0}^1))$ as in the above presentation.  Then the following relation holds in $\I(\Sigma_g)$:
\[ T_{c}^{2g_0-2}=   [\Push(\widetilde \beta_{g_0})^{-1}, \Push(\widetilde \alpha_{g_0})^{-1}] \cdots [\Push(\widetilde \beta_1)^{-1}, \Push(\widetilde \alpha_1)^{-1}].
\]
\end{lemma}

\subsection{Commutator lengths}

Let $G$ be a group, and consider its commutator subgroup $[G,G]$ endowed with the generating set consisting of all commutators in $G$.  We define the \emph{commutator length} of an element of $[G,G]$ to be the word length of that element with respect to this generating set.  In other words, the commutator length is the smallest number of commutators needed to write a given element of $[G,G]$.

Before we move on to the proof of Theorem~\ref{mainthm1} we will observe the following.

\begin{proposition} \label{CL}
Let $g \geq 3$ and let $c$ be a separating simple closed curve in $\Sigma_g$ that cuts off a genus 2 surface. Then the commutator length of $T_{c}^{2}$ in $\I(\Sigma_g)$ as well as in $\MCG(\Sigma_g)$ is 2. 
\end{proposition}

\begin{proof} 

Endo and Kotschick \cite{EndoKotschickCL} proved that the commutator length of $T_c^k$ in $\MCG(\Sigma_g)$ with $k > 0$ is bounded from below by
\[ \left\lceil 1+ \frac{k}{6(3g-1)} \right\rceil  \geq 2. \]
By Lemma~\ref{lemma:TorelliFactorization}, the commutator length of $T_c^2$ in $\I(\Sigma_g)$ is at most 2.  The proposition follows.
\end{proof}

\subsection{Indecomposable Torelli Lefschetz fibrations} We now prove Theorem~\ref{mainthm1}, which gives explicit Lefschetz fibrations with monodromy group in the Torelli group.

\begin{proof}[Proof of Theorem~\ref{mainthm1}]

We first treat the case $h=2$.  We will construct an explicit monodromy
\[ \pi_1(\Sigma_{2,2}) \to \I(\Sigma_g) \]
where $\Sigma_{2,2}$ is $\Sigma_2$ minus two points. 

By Lemma~\ref{lemma:TorelliFactorization} the following relation holds in $\I(\Sigma_{g})$:  
\[ T_c^2= [\Push(\widetilde \beta_2)^{-1}, \Push(\widetilde \alpha_2)^{-1}][\Push(\widetilde \beta_1)^{-1}, \Push(\widetilde \alpha_1)^{-1}]  \]
where the curves in the factorization are lifts of the curves in Figure~\ref{figure:generators}. This relation prescribes a genus $g$ Lefschetz fibration $(X,f)$ over $\Sigma_2$ whose monodromy group is contained in $\I(\Sigma_2)$. 

We would like to show that this Lefschetz fibration is fiber sum indecomposable. By the Fiber Sum Criterion, it suffices to show that the monodromy contains no essential simple loops in its kernel. 

We claim that the monodromy can be written as a composition:
\[ \pi_1(\Sigma_{2,2}) \stackrel{\eta}{\to} \pi_1(\UT(\Sigma_2)) \to \I(\Sigma_g). \]
We already described the second map and said that it is injective.  It remains to define $\eta$, to show that it agrees with our desired monodromy, and to show that the kernel of $\eta$ contains no essential simple separating loops.

Choose a standard generating set $\{\alpha_1,\beta_1,\alpha_2,\beta_2, \gamma_1,\gamma_2\}$ for $\pi_1(\Sigma_{2,2})$ as in Section~\ref{sec:monodromy}, so
\[ [\alpha_1, \beta_1] [\alpha_2, \beta_2] = \gamma_1 \gamma_2 \]
and so that the image of $\alpha_i$ in $\pi_1(\Sigma_2)$ equals the image of $\widetilde \alpha_i$ in $\pi_1(\Sigma_2)$ (similar for the $\beta_i$).

We define the map $\eta : \pi_1(\Sigma_{2,2}) \to \pi_1(\UT(\Sigma_2))$ on the generators:
\begin{align*}
\alpha_i &\mapsto \widetilde \alpha_i \\
\beta_i &\mapsto \widetilde \beta_i \\
\gamma_i &\mapsto t
\end{align*}
It follows from our presentation of $\pi_1(\UT(\Sigma_2))$ and the defining relation for our standard presentation of $\pi_1(\Sigma_2)$ that this map is a well-defined homomorphism.  It is clear that this map agrees with our original monodromy.

Our map $\eta$ makes the following diagram commute:
\[
\xymatrix{
\pi_1(\Sigma_{2,2}) \ar[dr]_\iota \ar[r]^\eta & \pi_1(\UT(\Sigma_2))  \ar[d] \\ 
 & \pi_1(\Sigma_2) && 
}
\]
where $\iota$ is induced by inclusion $\Sigma_{2,2} \to \Sigma_2$ and the vertical map is induced by the projection $\UT(\Sigma_2) \to \Sigma_2$. 

We claim that any essential simple loop $\delta$ in the kernel of $\iota$ is conjugate to
\[ (\delta_1 \delta_2)^{\pm 1}   \]
where each $\delta_i$ is conjugate to $\gamma_i$.  Indeed, any simple loop surrounding two punctures can be rewritten as a product of two simple loops, each surrounding one of the punctures, both in the same direction as the original loop.  Combining this with the fact that freely homotopic loops are conjugate in the fundamental group, the claim follows.

By the commutativity of the above diagram $\ker(\eta) \subseteq \ker(\iota)$, so it suffices to show $\eta(\delta)$ is nontrivial.  Since $\eta(\gamma_i)$ is the central element $t$ for all $j$ and $\delta_i$ is conjugate to $\gamma_i$, it follows that $\eta(\delta_i)=t$.  We conclude that $\eta(\delta)=t^{\pm 2}$.  We have thus succeeded in showing that the kernel of $\eta$ contains no essential simple loops, as desired.

\medskip

Now let $h \geq 3$.  Fix a covering map $p: \Sigma_h \to \Sigma_2$.  We would like to show that the pullback fibration is fiber sum indecomposable.  

The covering $p$ restricts to a covering map $p' : \Sigma_{h,2h-2} \to \Sigma_{2,2}$.  We consider the composition
\[ \pi_1(\Sigma_{h,2h-2}) \stackrel{p_\star'}{\to} \pi_1(\Sigma_{2,2}) \stackrel{\eta}{\to} \pi_1(\UT(\Sigma_2)) \to \I(\Sigma_{g}). \]

We now need to check that there are no essential separating loops in the kernel of $\eta \circ p_\star'$.  We treat two cases.

First assume that $\delta$ is an essential separating curve in $\Sigma_{h,2h-2}$ that does not lie in the kernel of $\iota' : \pi_1(\Sigma_{h,2h-2}) \to \pi_1(\Sigma_h)$.  There is a commutative diagram
\[
\xymatrix{
\pi_1(\Sigma_{h,2h-2}) \ar[r]^{\iota'} \ar[d]_{p'_\star} & \pi_1(\Sigma_h) \ar[d]^{p_\star} \\
\pi_1(\Sigma_{2,2}) \ar[r]^\iota & \pi_1(\Sigma_2) \\
}
\]
Since $p_\star$ is injective, the image of $\delta$ in $\pi_1(\Sigma_2)$ is nontrivial.  Therefore $p_\star'(\delta)$ does not lie in the kernel of $\iota$.  As above, this implies that $p_\star'(\delta)$ does not lie in the kernel of $\eta$.  Thus $\delta$ is not in the kernel of $\eta \circ p_\star'$.

Now suppose $\delta$ does lie in the kernel of $\iota'$.  As in the above claim, $\delta^{\pm 1}$ is a product of all counterclockwise simple peripheral loops.  Since $p'$ is the restriction of $p$, these all map to conjugates of the $\gamma_i$ in $\pi_1(\Sigma_{2,2})$ under $p_\star'$, and so then they all map to $t$ in $\pi_1(\UT(\Sigma_2))$ under $\eta$. The result now follows.

The Lefschetz fibrations we obtained in the proof are all relatively minimal since each loop in $\Sigma_{h,2h-2}$ surrounding a single puncture maps to $t^{\pm 1} \in \pi_1(\UT(\Sigma_2))$.  This completes the proof.
\end{proof}

As promised in the introduction, we also show that Theorem~\ref{mainthm1} cannot be extended to cover the cases of small fiber and base genera.

\begin{proposition} \label{SmallGenera}
When $g \leq 2$ or $h \leq 1$, there are no relatively minimal genus $g$ Lefschetz fibrations over a genus $h$ surface with monodromy group contained in the Torelli group. 
\end{proposition}

\begin{proof}
Let us begin with the case when the base genus $h$ is $0$ or $1$. As in Section~\ref{sec:monodromy}, a relatively minimal Lefschetz fibration can have monodromy group inside the Torelli group only if all the Dehn twists in its monodromy factorization are along nontrivial separating curves in the fiber. Since $h \leq 1$, the monodromy factorization of such a fibration would yield an expression of a product of nontrivial Dehn twists about separating curves as a product of at most one commutator in $\MCG(\Sigma_g)$. This contradicts the theorem of Endo and Kotschick we quoted in the proof of Proposition~\ref{CL} above. (An alternative proof for the $h=0$ case was already obtained by Ivan Smith \cite[Theorem~6.2]{SmithIndecomposable}.)

Next, for $g \leq 2$, the Torelli group $\I(\Sigma_g)$ is a free group (in fact, it is trivial when $g \leq 1$).  This precludes the possibility of finding a relation of the sort required for a monodromy factorization of a Lefschetz fibration, and so we are done.
\end{proof}


\subsection{Lefschetz fibrations that are not fiber sums of holomorphic ones}
Next is the proof of our third main theorem.  Specifically, we will give explicit constructions of Lefschetz fibrations that are not holomorphic, and moreover cannot be written as fiber sums of holomorphic fibrations.

\begin{proof}[Proof of Theorem~\ref{mainthm2}]

We will only present the desired Lefschetz fibrations over genus $2$ surfaces. The families of Lefschetz fibrations over surfaces of base genera $h \geq 3$ can then be obtained by taking pullbacks of these fibrations as in the proof of Theorem~\ref{mainthm1}. 

Let $c, \alpha_j^+, \alpha_j^-, \beta_j^+, \beta_j^-$ be the curves in $\Sigma_g$ given in Figure~\ref{figure:pushrelation}, and let $d$ be any nonseparating simple closed curve contained in the genus $g-2$ subsurface not shown in the picture.  The Dehn twist $T_d$ commutes with $T_{\alpha_1^+}$, $T_{\alpha_1^-}^{-1}$, $T_{\beta_1^+}$, and $T_{\beta_1^-}^{-1}$. It then follows from Lemma~\ref{lemma:TorelliFactorization} that the following relation holds in $\MCG(\Sigma_{g})$:
\[ T_{c}^2 \, [T_d^n T_{\alpha_1^-}T_{\alpha_1^+}^{-1} , T_{\beta_1^-}T_{\beta_1^+}^{-1} ] [T_{\alpha_2^-}T_{\alpha_2^+}^{-1} , T_{\beta_2^-}T_{\beta_2^+}^{-1}  ] = 1 .
\]

For each $n \geq 0$, let $(X_n, f_n)$ denote the genus $g$ Lefschetz fibration over a genus $2$ surface prescribed by the above factorization. Note that each one of these fibrations admits a section: the monodromy of our Lefschetz fibration over $\Sigma_2$ lifts to the mapping class group of a genus $g$ surface with one marked point, where this point is taken to lie in the component of $\Sigma_g-c$ containing $d$. The pullback fibrations similarly have monodromies factoring through the mapping class group of $\Sigma_g$ with one marked point, so they all admit sections. Also note that all these fibrations are relatively minimal, as indicated by their monodromy factorizations. When $n=0$, we have the Lefschetz fibrations constructed in the proof of Theorem~\ref{mainthm1}. 

We first claim that $(X_n, f_n)$ is fiber sum indecomposable.  There is a surjective homomorphism $\MCG(\Sigma_g,d) \to \MCG(\Sigma_g - d)$ with kernel $\langle T_d \rangle$, where $\MCG(\Sigma_g,d)$ is the subgroup of $\MCG(\Sigma_g)$ consisting of elements that preserve the isotopy class of $d$; see \cite[Proposition 3.20]{FarbMargalit}.  Since $f_0(\pi_1(\Sigma_{2,2}))$ lies in $\MCG(\Sigma_g,d)$ and intersects $\langle T_d \rangle$ trivially, it follows from Theorem~\ref{mainthm1} that the composition
\[ \pi_1(\Sigma_{2,2}) \stackrel{f_0}{\to} \MCG(\Sigma_g,d) \to \MCG(\Sigma_g - d) \]
contains no simple separating loops in its kernel.  But this composition is equal to the composition
\[ \pi_1(\Sigma_{2,2}) \stackrel{f_n}{\to} \MCG(\Sigma_g,d) \to \MCG(\Sigma_g - d). \]
Thus, $f_n$ contains no simple separating loops in its kernel. By the Fiber Sum Criterion, $(X_n, f_n)$ is fiber sum indecomposable.

We next check that $\{ (X_n,f_n) \mid n \in \Z^+\}$ consists of pairwise homotopy inequivalent $4$-manifolds.  We will do this by showing that their first homology groups are distinct. Since $(X_n,f_n)$ admits a section, we can calculate $H_1(X_n)$ using the action of $\pi_1(\Sigma)$ on $H_1(F)$ via the formula
\[ H_1(X_n) \cong H_1(\Sigma) \oplus \left( H_1(F)/\pi_1(\Sigma) \right). \]
As usual, we compute the latter quotient by fixing generators for $\pi_1(\Sigma)$ and $H_1(F)$. For $H_1(F)$ let us 
fix a standard basis $a_1, b_1, \ldots, a_{g}, b_{g}$ where $a_1=d$.  

Almost all of the mapping classes in the monodromy factorization for $f_n$ lie in the Torelli group.  The only nontrivial action of $\pi_1(\Sigma)$ on $H_1(F)$ comes from the $T_d^n$ factor in the monodromy image of the fundamental group generator corresponding to the first component of the commutator 
\[  [T_d^n T_{\alpha_1^-}T_{\alpha_1^+}^{-1} , T_{\beta_1^-}T_{\beta_1^+}^{-1} ].\]
 Furthermore, $T_d^n$ acts nontrivially only on one basis element of $H_1(F)$, namely $[b_1]$. We therefore obtain
\[ 
[b_1]= T_d^n([b_1]) = [b_1]-n \, \hat \imath (b_1,d)[d] = [b_1]+n[d] ,
\]
that is, $n[d]=0$ in in $H_1(X_n)$.  This is the only relation among the generators of $H_1(\Sigma)$ and $H_1(F)$. Hence 
\[ H_1(X_n) \cong \Z/n\Z \oplus \left(\bigoplus \Z^{2g+2h-1}\right). \] 
It is now immediate that the family $\{ X_n \mid n \in \Z^+\}$ consists of \linebreak $4$-manifolds that are pairwise homotopy inequivalent. 

Since $(X_n, f_n)$ are fiber sum indecomposable, to show that $(X_n,f_n)$ are not fiber sums of holomorphic fibrations, it suffices to show that $X_n$ cannot admit a complex structure.  As seen from our homology calculation, $b_1(X_n)$ is odd, and $X_n$ admits a relatively minimal Lefschetz fibration with fiber genus greater than one and base genus positive. The first author proved that in this situation, $X_n$ cannot support  a complex structure with either orientation \cite[Lemma 2]{BaykurHolomorphic}, proving our claim, and hence the theorem.
\end{proof}

A \emph{maximal section} of a Lefschetz fibration or a surface bundle is a section that attains the maximum possible self-intersection number over all sections. It was shown by the first author, Korkmaz, and Monden \cite{BaykurKorkmazMonden} that the maximal possible self-intersection number for a section of a genus $g$ Lefschetz fibration over a genus $h$ surface is $2h-2$ for $h \geq 1$, whereas this number was known to be $-1$ when $h=0$, provided $g \geq 2$. For each pair of integers $g \geq 2$ and $h \geq 1$, the authors moreover constructed examples of genus $g$ Lefschetz fibrations over genus $h$ surfaces with maximal sections. It can be seen from the monodromy factorizations of these bundles that they are indeed fiber sum decomposable, and infinite families of these examples are present in \cite{BaykurIndecomposable}. This therefore shows that the strategy of the proof of the analogous statement of Theorem~\ref{mainthm2} in the case of $h=0$ given by Stipsicz \cite{StipsiczIndecomposable} and Smith \cite{SmithIndecomposable} breaks down when $h \geq 1$, since their proof relied on an observation that a Lefschetz fibration over the $2$-sphere admitting a maximal section is always fiber sum indecomposable.

\section{Construction of surface bundles} \label{Sec: SBconstruction}

The main goal of this section is to prove Theorem~\ref{mainthm3}. That is, we construct explicit fiber sum indecomposable surface bundles over surfaces with monodromy group contained in the Torelli group.  Then we generalize this construction to the case where the monodromy group lies in any given term of the Johnson filtration of the Torelli group; see Theorem~\ref{theorem:filtration} below.

In \cite{BaykurMargalit}, we constructed fiber sum indecomposable surface bundles over surfaces by giving explicit injective, irreducible homomorphisms
\[ \pi_1(\Sigma_h) \to \MCG(\Sigma_g) \]
for any $g,h \geq 2$ as follows.  First, we defined families of homomorphisms:
\begin{align*}
\Phi_{h,n} &: \pi_1(\Sigma_h) \to A(\overline C_n) \hspace{5.5ex}  h \geq 2, n \geq 5, \\
\Psi_{n,p} &: A(\overline C_{n}) \to \B_{n} \hspace{9ex} n \geq 3,p \geq 3, \text{ and} \\
\Omega_g &: \B_{2g+1} \to \MCG(\Sigma_g) \hspace{4ex} g \geq 1,
\end{align*}
where $A(\overline C_n)$ is the group with presentation
\[ \langle v_1, \dots, v_n \mid [v_i,v_j] \iff |i-j| > 1 \mod n \rangle \]
and $B_n$ is the braid group on $n$ strands.   We regard $B_n$ as the mapping class group of a disk with $n$ marked points.

Our monodromies were given by
\[ \Omega_g \circ \Psi_{2g+1,p} \circ \Phi_{h,2g+1} : \pi_1(\Sigma_h) \to \MCG(\Sigma_g) \]
for $p \geq 3$, $g \geq 2$, and $h \geq 2$. 

The salient features of the constituent maps are:
\begin{enumerate}
\item $\Phi_{h,n}$ is injective for all $h \geq 2, n \geq 5$
\item $\Psi_{n,p}$ is injective for $n \geq 0, p \geq 3$
\item $\ker(\Omega_g) \leqslant Z(\B_{2g+1})$
\end{enumerate}
Applying these facts, plus the fact that $\pi_1(\Sigma_h)$ is centerless when $h \geq 2$, it immediately follows that $\Omega_g \circ \Psi_{2g+1,p} \circ \Phi_{h,2g+1}$ is injective.  

\medskip

The strategy for the proof of Theorem~\ref{mainthm3} is to modify the above construction by replacing $\Psi_{2g+1,p}$ with a different map.  In order to do this, we need to recount the definition of $\Psi_{n,p}$.  Recall that if $c$ is a simple closed curve in $\D_n$ surrounding exactly two marked points, then the half-twist $H_c$ is the unique mapping class whose square is the Dehn twist $T_c$.  With these in hand, $\Psi_{n,p}$ is given by
\[ \Psi_{n,p}(v_i) = H_{c_i}^p, \]
where each $c_i$ is the simple closed curve in $\D_n$ indicated in the left-hand side of Figure~\ref{figure:newmap}.

\begin{proof}[Proof of Theorem~\ref{mainthm3}]

We modify the above construction by replacing $\Phi_{h,2g+1}$ with $\Phi_{h,g+1}$ and $\Psi_{2g+1,p}$ with a map 
\[ \Psi_{g+1,p}' : A(\overline C_{g+1}) \to \B_{2g+1} \]
so that:
\begin{enumerate}
\item $\Psi_{2g+1,p}'$ is injective, and
\item the image of $\Omega_{2g+1} \circ \Psi_{g+1,p}'$ lies in $\I(\Sigma_g)$.
\end{enumerate}
It then follows that the monodromy
\[ \Omega_{2g+1} \circ \Psi_{g+1,p}' \circ \Phi_{h,g+1} \]
is injective with image in the Torelli group; the theorem then follows from the Fiber Sum Criterion.  Note the group $A(\overline C_n)$ only contains $\pi_1(\Sigma_2)$ when $n \geq 5$; hence our assumption that $g \geq 4$.

The basic idea for constructing $\Psi_{g+1,p}'$ is simply to insert one extra marked point interior to each $c_i$; see Figure~\ref{figure:newmap}.

\begin{figure}
\labellist
\small\hair 2pt
\pinlabel {$c_1$} [] at 42 40 
\pinlabel {$c_2$} [] at 70 40 
\pinlabel {$c_3$} [] at 108 40 
\pinlabel {$c_4$} [] at 142 40 
\pinlabel {$c_5$} [] at 135 73 
\pinlabel {$c_1$} [] at 234 40 
\pinlabel {$c_2$} [] at 262 40 
\pinlabel {$c_3$} [] at 300 40 
\pinlabel {$c_4$} [] at 334 40 
\pinlabel {$c_5$} [] at 327 73 
\endlabellist
\centering \includegraphics[scale=.9]{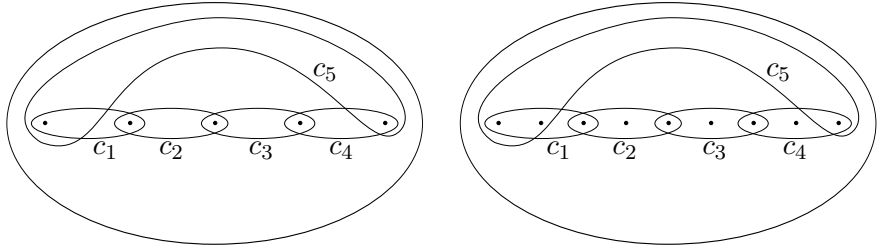}
\caption{The curves used to define $\Psi_{n,p}$ and $\Psi_{n,p}'$ in the case $n=5$.}
\label{figure:newmap}
\end{figure}

We define $\Psi_{n,p}' : A(\overline C_{n}) \to \B_{2n-1}$ by prescribing where it sends each generator $v_i$ for $A(\overline C_{n})$:
\[ v_i \mapsto T_{c_i}^{2p} \]
where the curves are as shown in the right-hand side of Figure~\ref{figure:newmap}.  We will now check that $\Psi_{n,p}'$ satisfies the properties listed above for $p \geq 1$.

First we show that $\Psi_{n,p}'$ is injective.  To this end, notice that the image of $\Psi_{n,p}'$ lies in the pure braid group $\PB_{2n-1}$. There is a forgetful homomorphism
\[ \PB_{2n-1} \to \PB_{n} \]
obtained by forgetting all of the even-numbered punctures (going from left to right).  The composition
\[ A(\overline C_{n}) \stackrel{\Psi_{n,p}'}{\to} \PB_{2n-1} \to \PB_{n} \leqslant \B_{n} \]
is nothing other than $\Psi_{n,4p}$, which is injective for $p \geq 1$.  It follows that $\Psi_{n,p}'$ is injective.

Next we show that $\Omega_{2g+1} \circ \Psi_{g+1}'$ lies in $\I(\Sigma_g)$.  Under the map $\Omega_g$, a square of a Dehn twist about an odd number of marked points in $\D_{2g+1}$ maps to a Dehn twist about a separating curve in $\Sigma_g$ \cite[Section 9.4.1]{FarbMargalit}.  But such Dehn twists lie in $\I(\Sigma_g)$, so we are done.
\end{proof}

As mentioned in the introduction, Theorem~\ref{mainthm3} is (at least for $g$ large enough) a special case of a more general phenomenon. Johnson defined \cite{JohnsonSurvey} a sequence of groups $\N_k(\Sigma_g)$ as follows.  The $k$th term is
\[ \N_k(\Sigma_g) = \ker(\MCG(\Sigma_g) \to \Out(\pi/\pi^{k+1})), \]
where $\pi = \pi_1(\Sigma_g)$, and $\pi^k$ is the $k$th term of the lower central series of $\pi$.  It follows from definitions that $\N_0(\Sigma_g) = \MCG(\Sigma_g)$ and that $\N_1(\Sigma_g) = \I(\Sigma_g)$.  It is a deep theorem of Johnson that $\N_2(\Sigma_g)$ is equal to $\K(\Sigma_g)$, the subgroup of $\MCG(\Sigma_g)$ generated by all Dehn twists about separating curves \cite{Johnson}.  The groups $\{\N_k(\Sigma_g)\}_{k \geq 1}$ form a central filtration of $\I(\Sigma_g)$, that is, successive quotients are abelian \cite{JohnsonSurvey}.  Also, it follows from work of Magnus \cite{Magnus} that these groups really do form a filtration:
\[ \bigcap_{k=0}^\infty \N_k(\Sigma_g) = 1. \]

We remark that the monodromy groups of the bundles we constructed in the proof of Theorem~\ref{mainthm3} are actually contained in $\N_2(\Sigma_g) = \K(\Sigma_g)$, which is an infinite index subgroup of $\I(\Sigma_g)$.

The following theorem says that we can construct surface bundles with monodromy group in any $\N_k(S_g)$.

\begin{theorem}
\label{theorem:filtration}
For $h \geq 2$, $g \geq 9$, and $k \geq 0$, there exist fiber sum indecomposable genus $g$ surface bundles over surfaces with monodromy in $\N_k(\Sigma_g)$.  
\end{theorem}

The existence of such bundles already appears in the unpublished work of Crisp and Farb \cite{CrispFarb}. Our main contribution is to give explicit bundles.  Without the explicitness, it is straightforward to use the same line of reasoning to obtain bundles under the same genus assumptions as Theorem~\ref{mainthm3}.

\begin{proof}[Proof of Theorem~\ref{theorem:filtration}]

For simplicity, we deal with the case $g$ odd.  The case of $g$ even can be handled similarly.  Also, as usual, we restrict to the case $h=2$; the case of $h > 2$ is handled by pulling back via any covering map $\Sigma_h \to \Sigma_2$.

We will again modify the previous constructions by introducing a new map
\[ \Xi_{g,k} : A(\overline C_{(g+1)/2}) \to A(\overline C_{g+1}).\]
To define the map, we need some notation.  For two elements $g_1,g_2$ of a group, define group elements $c_i(g_1,g_2)$ as follows:
\[ c_1(g_1,g_2) = g_2 \qquad c_{i+1}(g_1,g_2) = [g_1,c_i]. \]
By design, $c_k(g_1,g_2)$ is contained in the $k$th term of the lower central series of the ambient group.  We now define
\[ \Xi_{g,k}(v_i) = c_k(v_{2i-1},v_{2i}). \]
This map is injective for $k \geq 2$ by Kim's co-contraction theorem; see \cite[Theorem 2]{BaykurMargalit} or \cite[Remark 5.2]{Kim}.

Our monodromies will be given by:
\[ \Omega_{2g+1} \circ \Psi_{g+1,p}' \circ \Xi_{g,k} \circ \Phi_{h,(g+1)/2}. \]
Again, $\Phi_{h,(g+1)/2}$ is only defined for $(g+1)/2 \geq 5$, hence the assumption $g \geq 9$.  Since the new map $\Xi_{g,k}$ is injective, it follows that this monodromy is injective.  It remains to show that the image of the composition lies in $\N_k(\Sigma_g)$.  Indeed, since the image of $\Omega_{2g+1} \circ \Psi_{g+1,p}'$ lies in $\I(\Sigma_g)$, and the image of $\Xi_{g,k}$ lies in the $k$th term of the lower central series of $A(\overline C_{g+1})$, it follows that the image of $\Omega_{2g+1} \circ \Psi_{g+1,p}' \circ \Xi_{g,k}$ lies in $\I(\Sigma_g)^k$, the $k$th term of the lower central series for $\I(\Sigma_g)$.  But $\I(\Sigma_g)^k$ is contained in $\N_k(\Sigma_g)$, as the $k$th term of any central series for $\I(\Sigma_g)$ contains $\I(\Sigma_g)^k$.  This completes the proof.
\end{proof}

We remark that the surface bundles we construct in the proof of Theorems~\ref{mainthm3} and~\ref{theorem:filtration} are in fact section sum indecomposable as in \cite{BaykurMargalit}. The proof is essentially the same as for the examples given in that paper, namely, it suffices to show that the monodromy groups are irreducible.  The key point is that the curves $c_i$ in $\D_{2g+1}$ are still filling in the new construction, and so we can again use Penner's criterion for irreducibility in the mapping class group.

We also note that our constructions in the proofs of Theorems~\ref{mainthm3} and~\ref{theorem:filtration} can be easily modified in order to obtain infinitely many such bundles, distinct up to fiberwise diffeomorphism.

Finally, one might wonder if there are Lefschetz fibrations with monodromy in $\N_k(\Sigma_g)$ for $k \geq 2$. However, when $k \geq 3$, the group $\N_k(\Sigma_g)$ contains no Dehn twists (or even multitwists) \cite[Theorem A.1]{BBM}, so there are no such fibrations.


\bibliographystyle{plain}
\bibliography{LefschetzTorelli}

\end{document}